%% file: papier.tex
\documentclass[a4paper]{amsart}

\usepackage{amsmath,amssymb,amscd}
\usepackage[arrow,matrix]{xy}
 \objectmargin={6pt}
\usepackage{enumerate}
\usepackage[dvipdfm]{graphicx}
\usepackage{cite}

\theoremstyle{change}
\newtheorem{thm}{Theorem}[section]

\newtheorem{prop}[thm]{Proposition}
\newtheorem{cor}[thm]{Corollary}
\newtheorem{lem}[thm]{Lemma}

\newtheorem{clm}[equation]{}
\newtheorem{slem}[equation]{}
\theoremstyle{definition}
\newtheorem{rem}[thm]{Remark}
\numberwithin{equation}{section}

\def\det{\mathop{\mathrm{det}}\nolimits}
\def\Im{\mathop{\mathrm{Im}}\nolimits}

\def\Coker{\mathop{\mathrm{Coker}}\nolimits}

\def\Stab{\mathop{\mathrm{Stab}}\nolimits}

\def\so{\mathop{\mathrm{SO}}\nolimits}
\def\o{\mathop{\mathrm{O}}\nolimits}
\def\inc{\mathop{\mathrm{inc}}\nolimits}
\def\id{\mathop{\mathrm{id}}\nolimits}
\def\proj{\mathop{\mathrm{proj}}\nolimits}
\def\tchar{\mathop{\mathrm{char}}\nolimits}

\newcommand{\bb}[1]{{\mathbb{#1}}}
\newcommand{\mca}[1]{{\mathcal{#1}}}
\newcommand{\acts}{\!\cdot\! }

\title[On homological stability for $\mathrm{O}_n$ and $\so_n$]
{On homological stability for orthogonal groups and special orthogonal groups}
\author{Masayuki Nakada}
\address{Department~of~Mathematics, Graduate~School~of~Science, Kyoto~University, Kyoto~606-8502, Japan}
\email{masayuki@math.kyoto-u.ac.jp}
\date{\today}

\begin{document}

\begin{abstract}
The problem of homological stability
helps us to catch the structure of group homology.
We calculate homological stability of
special orthogonal groups, and we also calculate the stability of orthogonal groups
with determinant-twisted coefficients under a certain good situation.
We also get some results about the structure of these homology.

\end{abstract}

\maketitle

\tableofcontents

\section{Introduction}
\input{papier-1}

\section{Involution $\sigma$}\label{involsec}
\input{papier-2}

\section{Induction algorithm}\label{specseq}
\input{papier-3}

\section{On structures of $\sigma$-coinvariant part}\label{strsec}
\input{papier-4}

\section{Proof of theorem \ref{mainthm1}}\label{proofofthm1}
\input{papier-5}

\section{Applications}
\input{papier-6}

\begin{appendix}
\end{appendix}

\bibliographystyle{plain}
\bibliography{ref}

\end{document}

%% file: papier-1.tex
\subsection{Notations and Preliminaries}

In this paper, $F$ is an infinite Pythagorean field of $\tchar(F)\neq 2$.
A field $F$ is Pythagorean if the sum of any two squares is a square.
Other than algebraically closed fields, the real numbers $\bb{R}$ is a typical example.
Let $q(x)=\sum_{i=1}^n x_i^2$ be the Euclidean quadratic form on $F^n$.
We denote by $\o_n=\o_n(F,q)$ the corresponding orthogonal group
and by $\so_n=\so_n(F,q)$ the corresponding special orthogonal group of degree $n$.
We denote by $S=S(F^n)$ the unit sphere $\{x\in F^n; q(x)=1\}$.
We write by $\bb{Z}^t$ the determinant-twisted $\o_n$-module
which admits the twisted action by the determinant.

We consider, for any integer $n\geq 0$, that $F^n$ is isometrically embedded in $F^{n+1}$ as $x\mapsto(0,x)$.
This defines an inclusion map
\begin{equation}\label{defofinc}
\inc\colon\o_n\rightarrow\o_{n+1},\qquad A\mapsto\begin{pmatrix}1 & 0\\ 0 & A\end{pmatrix}
\end{equation}
for $n\geq 1$ and, in case $n=0$, we consider $\o_0$ as the trivial group.
The inclusion (\ref{defofinc}) induces the map of homology groups
\begin{equation}
H_i(\inc)\colon H_i(\o_n)\longrightarrow H_i(\o_{n+1})
\end{equation}
in the $i$-th degree.
We understand the coefficient of homology is $\bb{Z}$ if it is omitted.
It is well known that every isomorphic embedding induces an inclusion $\inc\colon\o_n\hookrightarrow\o_{n+1}$
which are conjugate each other by the theorem of Witt and induces the same map in homology.
In the same way, from the inclusion $\inc\colon\so_n\hookrightarrow\so_{n+1}$
we have $H_i(\inc)\colon H_i(\so_n)\rightarrow H_i(\so_{n+1})$.

$H_i(\so_n)$ admits an involution induced from the short exact sequence
\begin{equation}\label{ext}
\xymatrix{
1 \ar[r] & \so_n \ar[r]^i & \o_n \ar[r]^{\det} & \bb{Z}/2 \ar[r] & 1.
}
\end{equation}
Here $\bb{Z}/2$ means the multiplicative group $\{\pm 1\}$,
$i$ is the natural inclusion and $\det$ is the determinant homomorphism.
We denote by $\sigma$ this involution.

\subsection{Results}

The problem of homological stability of $\o_n$ was first studied by Sah in \cite{sah1986} in case $F=\bb{R}$.
Cathelineau generalised the result of Sah for any infinite Pythagorean fields in \cite{cat2007}.
In \cite{cat2007}, Cathelineau proved the following;
let $\o_n=\o_n(F,q)$ be the orthogonal group over an infinite Pythagorean field $F$ with Euclidean quadratic form $q$.
\begin{prop}[\cite{cat2007,sah1986}]\label{catprop}
The map $H_i(\inc)\colon H_i(\o_n)\rightarrow H_i(\o_{n+1})$ is bijective for $i<n$,
and surjective for $i\leq n$.
\end{prop}

In Sah's paper \cite{sah1986} only $H_2$ of the stability for $\so_n(\bb{R})$ was studied.
In \cite{cat2007} Cathelineau proved the following result for $\so_n=\so_n(F,q)$:
\begin{prop}[\cite{cat2007}]\label{catprop2}
The map $H_i(\inc)\colon H_i(\so_n,\bb{Z}[1/2])\rightarrow H_i(\so_{n+1},\bb{Z}[1/2])$
is bijective for $2i<n$, and surjective for $2i\leq n$.
\end{prop}
In case $F$ is quadratically closed, it is known that the obstruction to stability for $\so_n$
with coefficient $\bb{Z}[1/2]$ is the Milnor $K$ group $K_n^M(F)$ \cite{cat2007}.

We make precise the result of Cathelineau's result (proposition \ref{catprop2}).
Our result for special orthogonal groups is the following:
\begin{thm}\label{mainthm1}
The map $H_i(\inc)\colon H_i(\so_n)\rightarrow H_i(\so_{n+1})$ is bijective for $2i<n$,
and surjective for $2i\leq n$.
\end{thm}

In the proof of theorem \ref{mainthm1} we also get the following corollary (see section \ref{proofofthm1}).
\begin{cor}\label{strofson}
For $2i<n$, the group $H_i(\so_n)$ is isomorphic to its own $\sigma$-inariant part.
\end{cor}

Another implication is;
\begin{thm}\label{mainthm2}
The map $H_i(\inc,\bb{Z}^t)\colon H_i(\o_n,\bb{Z}^t)\rightarrow H_i(\o_{n+1},\bb{Z}^t)$
is bijective for $2i<n$, and surjective for $2i\leq n$.
\end{thm}
The group $H_n(\o_n,\bb{Z}^t)$ plays an important role in the problem of spherical scissors congruence (see \cite{dupont2001}).

\subsubsection*{Acknowledgements}
The author would like to thank Masana Harada for his helpful supports.

%% file: papier-2.tex
In this section we study the involution on $H_i(\so_n)$ induced from the extension (\ref{ext}).

Let $R$ be a commutative ring, $G$ be a group and $M$ be a left $RG$-module.
Let $\gamma$ be an element in $G$.
We can define an endomorphism $\gamma_{\ast}$ on $H_i(G,M)$ by
\begin{equation*}
\gamma_{\ast}([g_1|\ldots|g_i]\otimes m)=[\gamma g_1\gamma^{-1}|\ldots|\gamma g_i\gamma^{-1}]\otimes \gamma m
\end{equation*}
using the standard bar resolution.
Here $g_1,\ldots,g_i\in G$ and $m\in M$.
Notice that
\begin{equation}\label{invact}
\text{the endomorphism $\gamma_{\ast}$ is chain homotopic to the identity.}
\end{equation}

For the proof, see \cite[Lemma 5.4]{dupont2001} for instance.

The sequence (\ref{ext}) splits by
\begin{equation*}
\iota \colon \bb{Z}/2 \rightarrow \o_n,\qquad \pm 1 \mapsto \begin{pmatrix}\pm 1 & 0 \\ 0 & 1_{n-1}\end{pmatrix},
\end{equation*}
where $1_{n-1}$ means the $(n-1)$-unit matrix.

The sequence (\ref{ext}) induces the action of $\bb{Z}/2=\{ \pm 1\}$ on $\so_n$.
The $(-1)$-action
\begin{equation*}
-1\acts g:= \begin{pmatrix}-1 & 0 \\ 0 & 1_{n-1}\end{pmatrix} g \begin{pmatrix}-1 & 0 \\ 0 & 1_{n-1}\end{pmatrix}^{-1}
\end{equation*}
defines an involution $\sigma$ on $H_i(\so_n)$.

Observe that the following diagram
\begin{equation}\label{commlemma}
\xymatrix{
H_i(\so_n) \ar[d]_{\sigma} \ar[r]^{H_i(\inc)} & H_i(\so_{n+1}) \ar[d]^{\sigma} \\
H_i(\so_n) \ar[r]_{H_i(\inc)} & H_i(\so_{n+1})
}
\end{equation}
is commutative, for{\small
\begin{eqnarray*}
&&\inc(\sigma\acts g)\\
&=&\begin{pmatrix}1 & 0 \\ 0 & \begin{pmatrix}-1 & 0 \\ 0 & 1_{n-1}\end{pmatrix}g\begin{pmatrix}-1 & 0 \\ 0 & 1_{n-1}\end{pmatrix}^{-1}\end{pmatrix}\\
&=&\begin{pmatrix}(-1)^{-1} & 0 & 0 \\ 0 & -1 & 0 \\ 0 & 0 & 1_{n-1}\end{pmatrix}\begin{pmatrix}-1 & 0 & 0 \\ 0 & 1 & 0 \\ 0 & 0 & 1_{n-1} \end{pmatrix}\begin{pmatrix}1 & 0 \\ 0 & g \end{pmatrix}
\begin{pmatrix}(-1)^{-1} & 0 & 0 \\ 0 & 1 & 0 \\ 0 & 0 & 1_{n-1} \end{pmatrix}\begin{pmatrix}-1 & 0 & 0 \\ 0 & (-1)^{-1} & 0 \\ 0 & 0 & 1_{n-1}\end{pmatrix}\\
&=&\begin{pmatrix}(-1)^{-1} & 0 & 0 \\ 0 & -1 & 0 \\ 0 & 0 & 1_{n-1}\end{pmatrix}\sigma\acts(\inc(g))\begin{pmatrix}(-1)^{-1} & 0 & 0 \\ 0 & -1 & 0 \\ 0 & 0 & 1_{n-1}\end{pmatrix}^{-1},
\end{eqnarray*}
}
where $\begin{pmatrix}(-1)^{-1} & 0 & 0 \\ 0 & -1 & 0 \\ 0 & 0 & 1_{n-1}\end{pmatrix}\in\so_{n+1}$, so that
this action on homology $H_i(\so_{n+1})$ is trivial.
Moreover we have
\begin{equation}\label{corinvty}
\Im H_i(\inc)\subseteq H_i(\so_{n+1})^{\sigma},
\end{equation}
for
\begin{equation*}
-1\acts(\inc(g))=-1\acts\begin{pmatrix}1 & 0 \\ 0 & g\end{pmatrix}=\begin{pmatrix}-1 & 0 \\ 0 & 1_n\end{pmatrix}\begin{pmatrix}1 & 0 \\ 0 & g\end{pmatrix}\begin{pmatrix}-1 & 0 \\ 0 & 1_n\end{pmatrix}^{-1}=\begin{pmatrix}1 & 0 \\ 0 & g\end{pmatrix}=\inc(g),
\end{equation*}
where $H_i(\so_{n+1})^{\sigma}$ is $\sigma$-invariant part of $H_i(\so_n)$.

From the above argument, there exists a map
\begin{equation}
H_i(\inc)_{\sigma}^{\prime}\colon H_i(\so_n)_{\sigma}\rightarrow H_i(\so_{n+1}),
\end{equation}
where we denote by $H_i(\so_n)_{\sigma}$ $\sigma$-coinvariant part of $H_i(\so_n)$,
so that the diagram
\begin{equation}
\xymatrix{
H_i(\so_n) \ar[d]_{\pi} \ar[r]^{H_i(\inc)}& H_i(\so_{n+1}) \ar[d]^{\pi} \\
H_i(\so_n)_{\sigma} \ar[ur]^{H_i(\inc)_{\sigma}^{\prime}} \ar[r]_{H_i(\inc)_{\sigma}} & H_i(\so_{n+1})_{\sigma}
}
\end{equation}
is commutative, where the map $\pi$ is the natural projection,
and the map $H_i(\inc)_{\sigma}$ factors as $\pi\circ H_i(\inc)_{\sigma}^{\prime}$.

%% file: papier-3.tex
In this section we shall prove the following inductive statement on $q$ for each fixed $n$ and $i$:
\begin{lem}\label{lemA}
\begin{equation}\label{soncoinvstab3}
H_q(\so_n)_{\sigma}\rightarrow H_q(\so_{n+1})_{\sigma}\text{ is bijective for $q\leq i$}
\end{equation}
if the following two conditions are satisfied;
\begin{equation}\label{onstab3}
H_q(\o_n)\rightarrow H_q(\o_{n+1})\quad \text{is bijective for $q\leq i$}
\end{equation}
and 
\begin{equation}\label{sonstab3}
H_q(\so_n)\rightarrow H_q(\so_{n+1})\quad \text{is bijective for $q<i$}.
\end{equation}
\end{lem}

First we compare the Lyndon-Hochschild-Serre spectral sequences (see for instance \cite{McCleary2001}) on $\o_n$.
Let $E^r$ be the $r$-th term of the Lyndon-Hochschild-Serre homology spectral sequence
associated to \eqref{ext}, and $\tilde{E}^r$ be that on $\o_{n+1}$.

Since we use the bar resolutions, $E^1_{1,q}$ can be seen to be generated by the classes of the form $[a]\otimes c$,
where $a$ is an element of $\bb{Z}/2$ and $c\in H_q(\so_n)$.
We denote by $\delta$ the horizontal map of the double bar complex which induces $E^2$.
Then from the definition of $\delta$, it holds that
\begin{equation*}
\delta([a]\otimes c)=a\acts c-c.
\end{equation*}
Hence its class is zero in $\bb{Z}/2$-coinvariant part $H_0(\bb{Z}/2,H_q(\so_n))$.

Thus we get that for any $q$, $E_{0,q}^2=H_0(\bb{Z}/2,H_q(\so_n))$ injects into $E^{\infty}_{0,q}$,
so that
\begin{equation}\label{lem2}
E_{0,q}^2=E_{0,q}^3=\cdots=E_{0,q}^{\infty}.
\end{equation}
The same is true for $\tilde{E}^r_{0,q}$.
Hence we have from (\ref{lem2}) that there exists a natural injection $H_i(\so_n)_{\sigma}\rightarrow H_i(\o_n)$.

We denote by $\{\mca{F}_p\}$ the filtration on the $i$-th homology $H_i(\o_n)$
induced by the Lyndon-Hochschild-Serre spectral sequence obtained from (\ref{ext}),
and by $\{\tilde{\mca{F}}_p\}$ that on $H_i(\o_{n+1})$.
We write the filtration on $H_i(\o_n)$ by
\begin{equation*}
0=\mca{F}_{-1}\subseteq \mca{F}_0\subseteq \mca{F}_1\subseteq \cdots\subseteq \mca{F}_p\subseteq\cdots\subseteq H_i(\o_n).
\end{equation*}
Here we treat the filtration only on the $i$-th homology, so that this filtration is of length $i$ on $H_i(\o_n)$.

Recall that if $p+q=i$ then
\begin{equation*}
E_{p,q}^{\infty}=\mca{F}_p/\mca{F}_{p-1}.
\end{equation*}

First we compare the `left' parts.
\begin{slem}\label{sublemofA}
Under the condition that \eqref{onstab3} and \eqref{sonstab3} are satisfied, the map
\begin{equation*}
H_i(\o_n)/\mca{F}_0\rightarrow H_i(\o_{n+1})/\tilde{\mca{F}}_0
\end{equation*}
is bijective.
\end{slem}

To prove \ref{sublemofA}, first we check the following \ref{lem3}.
\begin{clm}\label{lem3}
If both the condition \eqref{onstab3} and \eqref{sonstab3} are satisfied,
the following map of spectral sequences
\begin{equation*}
E^{\infty}_{p,q}(\inc)\colon E_{p,q}^{\infty}\rightarrow\tilde{E}_{p,q}^{\infty}
\end{equation*}
is bijective for $p+q\leq i$ and $q<i$.
\end{clm}

\begin{proof}
First we calculate $E^2$-terms.
The condition says that for $q<i$, the map
\begin{equation*}
E^2_{p,q}(\inc)\colon E^2_{p,q}\rightarrow \tilde{E}^2_{p,q}
\end{equation*}
is bijective.
Moreover, let $d^2_{p,q}\colon E^2_{p,q}\rightarrow E^2_{p-2,q+1}$ and
$\tilde{d}^2_{p,q}\colon \tilde{E}^2_{p,q}\rightarrow \tilde{E}^2_{p-2,q+1}$ be
the corresponding differentials, then we get
\begin{equation*}
\tilde{d}^2_{p,q}\circ E^2_{p,q}(\inc)=E^2_{p-2,q+1}(\inc)\circ d^2_{p,q}
\end{equation*}
for $q<i-1$, and from (\ref{lem2}) we obtain
\begin{equation*}
\tilde{d}^2_{2,i-1}\circ E^2_{2,i-1}(\inc)=0=E^2_{0,i}(\inc)\circ d^2_{2,i-1}.
\end{equation*}
Hence, at $E^3$-terms
\begin{equation*}
E^3_{p,q}(\inc)\colon E^3_{p,q}\rightarrow \tilde{E}^3_{p,q}
\end{equation*}
is bijective for $p+q\leq i$ and $q<i$, or for $q<i-1$ and any $p$.
Moreover, we also obtain that
\begin{equation*}
\tilde{d}^3_{p,q}\circ E^3_{p,q}(\inc)=E^3_{p-3,q+2}(\inc)\circ d^3_{p,q}
\end{equation*}
for $q<i-2$ and that
\begin{equation*}
\tilde{d}^3_{3,i-2}\circ E^3_{3,i-2}(\inc)=0=E^3_{0,i}(\inc)\circ d^3_{3,i-2}.
\end{equation*}

Repeating this process until $r=i+1$, both $E_{p,q}^r$ and $\tilde{E}^r_{p,q}$ are
degenerate at $p+q\leq i$, and therefore for $p+q\leq i$ and $q<i$ the map
\begin{equation*}
E^{\infty}_{p,q}(\inc)\colon E^{\infty}_{p,q}\rightarrow \tilde{E}^{\infty}_{p,q}
\end{equation*}
is bijective.
\end{proof}

From \ref{lem3}, using five lemmas repeatingly, we have the following:
\begin{slem}\label{lem4}
If the condition \eqref{onstab3} and \eqref{sonstab3} are satisfied, the natural map
\begin{equation*}
F_p\colon \mca{F}_p/\mca{F}_0\rightarrow\tilde{\mca{F}}_p/\tilde{\mca{F}}_0
\end{equation*}
is bijective.
\end{slem}

Now we begin the proof of lemma \ref{lemA}.

Suppose that 
\begin{equation*}
H_q(\inc) \colon H_q(\o_n)\rightarrow H_q(\o_{n+1})
\end{equation*}
is bijective for $q\leq i$ and
\begin{equation*}
H_q(\inc) \colon H_q(\so_n) \rightarrow H_q(\so_{n+1})
\end{equation*}
is bijective for $q<i$.
Then by \ref{lem4}, the map
\begin{equation*}
H_q(\inc)/\mca{F}_0\colon H_q(\o_n)/\mca{F}_0\rightarrow H_q(\o_{n+1})/\tilde{\mca{F}}_0
\end{equation*}
becomes bijective.
Notice that $H_i(\so_n)_{\sigma}\cong\mca{F}_0$, and $H_i(\so_{n+1})_{\sigma}\cong\tilde{\mca{F}}_0$ respectively.

Hence the middle and right vertical maps in the diagram
\begin{equation*}
\xymatrix{
0\ar[r] & H_i(\so_n)_{\sigma}\ar[r]\ar[d]_{H_i(\inc)_{\sigma}} & H_i(\o_n)\ar[r]\ar[d]_{H_i(\inc)} & H_i(\o_n)/\mca{F}_0\ar[r]\ar[d]_{H_i(\inc)/\mca{F}_0} & 0\\
0\ar[r] & H_i(\so_{n+1})_{\sigma}\ar[r] & H_i(\o_{n+1})\ar[r] & H_i(\o_{n+1})/\tilde{\mca{F}}_0\ar[r] & 0
}
\end{equation*}
(where the injectivities of left arrows are from (\ref{lem2})) are bijective and the horizontal columns in the above are exact.
That is, using five lemma, the left vertical map is also bijective.
This concludes the claim of lemma \ref{lemA}.
\qed

We define the map $\rho_i$ as the composition of natural inclusion and projection:
\begin{equation}\label{rho_i}
\rho_i\colon H_i(\so_n)^{\sigma}\rightarrow H_i(\so_n)\rightarrow H_i(\so_n)_{\sigma}.
\end{equation}
Let us see the following commutative diagram
\begin{equation}\label{diag***}
\xymatrix{
H_i(\so_n)^{\sigma}\ar[r]^{\rho_i}\ar[d]_{H_i(\inc)^{\sigma}} & H_i(\so_n)_{\sigma}\ar[r]\ar[d]^{H_i(\inc)_{\sigma}}\ar[dl]_{H_i(\inc)_{\sigma}^{\prime}} & H_i(\so_n)_{\sigma}/\Im\rho_{n,i}\ar[r]\ar[d] & 0\\
H_i(\so_{n+1})^{\sigma}\ar[r]_{\rho_i} & H_i(\so_{n+1})_{\sigma}\ar[r] & H_i(\so_{n+1})_{\sigma}/\Im\rho_i\ar[r] & 0
}
\end{equation}
where the left vertical map $H_i(\inc)^{\sigma}$ means the restriction.
Here if $H_i(\inc)_{\sigma}$ is surjective then, using \eqref{diag***},
$\rho_i\colon H_i(\so_{n+1})^{\sigma}\rightarrow H_i(\so_{n+1})_{\sigma}$ in \eqref{diag***} is surjective,
and hence
\begin{equation*}
\Coker\rho_i=H_i(\so_{n+1})_{\sigma}/\Im\rho_i
\end{equation*}
is zero.

We have proved the following proposition.

\begin{lem}\label{lem11}
If $H_i(\inc)_{\sigma}\colon H_i(\so_n)_{\sigma}\rightarrow H_i(\so_{n+1})_{\sigma}$ is surjective, then
\begin{equation*}
\Coker\{\rho_i\colon H_i(\so_{n+1})^{\sigma}\rightarrow H_i(\so_{n+1})_{\sigma}\}=0.
\end{equation*}
\end{lem}

%% file: papier-4.tex
In this section we prove the following lemma:
\begin{lem}\label{lemB}
For all $n$ and all $i$, the following isomorphism exists:
\begin{equation}
(H_i(\so_n)/H_i(\so_n)^{\sigma})\otimes_{\bb{Z}}\bb{Z}/2\cong\Coker\rho_i.
\end{equation}
\end{lem}

Let us first notice that the involution $\sigma\colon H_i(\so_n)\rightarrow H_i(\so_n)$
induces an involution endomorphism on $H_i(\so_n)/H_i(\so_n)^{\sigma}$.

\begin{slem}\label{lem8}
There exists an isomorphism
\begin{equation*}
\Coker\rho_i\cong(H_i(\so_n)/H_i(\so_n)^{\sigma})_{\sigma},
\end{equation*}
where $\rho_i$ is the map defined in \eqref{rho_i}.
\end{slem}

For the left exactness of the functor $\bullet\otimes_{\bb{Z}[\bb{Z}/2]}\bb{Z}$ of taking $\sigma$-coinvariant part,
we get the following commutative diagram
\begin{equation}
\xymatrix@C-13pt{
0\ar[r] & H_i(\so_n)^{\sigma} \ar[r] \ar[d] \ar[dr]^{\rho_i} & H_i(\so_n) \ar[r]^{\proj} \ar[d] & \frac{H_i(\so_n)}{H_i(\so_n)^{\sigma}} \ar[r] \ar[d] & 0\\
 & H_i(\so_n)^{\sigma}\otimes_{\bb{Z}[\bb{Z}/2]}\bb{Z} \ar[r] & H_i(\so_n)\otimes_{\bb{Z}[\bb{Z}/2]}\bb{Z} \ar[r]_{\proj_{\sigma}} & \left(\frac{H_i(\so_n)}{H_i(\so_n)^{\sigma}}\right)\otimes_{\bb{Z}[\bb{Z}/2]}\bb{Z} \ar[r] & 0
}
\end{equation}
where $\proj$ is the natural projection and $\proj_{\sigma}=\proj\otimes_{\bb{Z}[\bb{Z}/2]}\bb{Z}$, horizontal sequences are exact, and each vertical map sends $x$ to $x\otimes 1$.

Next we see that each element in the module $(H_i(\so_n)/H_i(\so_n)^{\sigma})_{\sigma}$ is annihilated by $2$.

\begin{clm}\label{lem9}
There exists the isomorphism
\begin{equation*}
\left(\frac{H_i(\so_n)}{H_i(\so_n)^{\sigma}}\right)_{\sigma}\cong\frac{H_i(\so_n)}{H_i(\so_n)^{\sigma}}\bigg/2\frac{H_i(\so_n)}{H_i(\so_n)^{\sigma}}.
\end{equation*}
\end{clm}

Generally, from each left $\bb{Z}$-module $M$ with involution $\sigma$ we have the short exact sequence
\begin{equation}
\xymatrix{
0\ar[r] & M^{\sigma}\ar[r]^{\text{inclusion}} & M\ar[r]^{\tau} & M\ar[r]^{\proj} & M_{\sigma}\ar[r] & 0
}
\end{equation}
Here $\tau=\id-\sigma$.
In case $M=H_i(\so_n)/H_i(\so_n)^{\sigma}$, we get the following exact sequence
\begin{equation}\label{exseq}
0\rightarrow\left(\frac{H_i(\so_n)}{H_i(\so_n)^{\sigma}}\right)^{\sigma}\rightarrow\frac{H_i(\so_n)}{H_i(\so_n)^{\sigma}}\xrightarrow{\tau}\frac{H_i(\so_n)}{H_i(\so_n)^{\sigma}}\rightarrow\left(\frac{H_i(\so_n)}{H_i(\so_n)^{\sigma}}\right)_{\sigma}\rightarrow 0.
\end{equation}

Observe that $\tau\colon H_i(\so_n)/H_i(\so_n)^{\sigma}\rightarrow H_i(\so_n)/H_i(\so_n)^{\sigma}$
is multiplication by $2$.
For this one may use the fact that any $a\in H_i(\so_n)$ we have $(1+\sigma)a\in H_i(\so_n)^{\sigma}$.
This is because
\begin{equation*}
\sigma(1+\sigma)a=(\sigma+\sigma^2)a=(\sigma+1)a=(1+\sigma)a.
\end{equation*}
Hence if we denote the class of $a$ in $H_i(\so_n)/H_i(\so_n)^{\sigma}$ by $[a]$,
\begin{equation*}
\tau [a]=(1-\sigma)[a]=[(1-\sigma)a]+[(1+\sigma)a]=[(1-\sigma)a+(1+\sigma)a]=[2a]=2[a]
\end{equation*}
in $H_i(\so_n)/H_i(\so_n)^{\sigma}$.

Therefore using the exact sequence \ref{exseq} we have
\begin{eqnarray*}
\left(\frac{H_i(\so_n)}{H_i(\so_n)^{\sigma}}\right)_{\sigma}&\cong&\Coker\tau\\
&=&\frac{H_i(\so_n)}{H_i(\so_n)^{\sigma}}\bigg/\Im\tau\\
&=&\frac{H_i(\so_n)}{H_i(\so_n)^{\sigma}}\bigg/2\frac{H_i(\so_n)}{H_i(\so_n)^{\sigma}}.
\end{eqnarray*}

Now if we use the claims of \ref{lem8} and \ref{lem9} we have the following isomorphisms:
\begin{equation*}
\frac{H_i(\so_n)_{\sigma}}{\Im\rho_i}\cong\left(\frac{H_i(\so_n)}{H_i(\so_n)^{\sigma}}\right)_{\sigma}\cong\frac{H_i(\so_n)}{H_i(\so_n)^{\sigma}}\bigg/2\frac{H_i(\so_n)}{H_i(\so_n)^{\sigma}},
\end{equation*}
and hence we get lemma \ref{lemB}.

%% file: papier-5.tex
If $H_i(\inc)_{\sigma}\colon H_i(\so_n)_{\sigma}\rightarrow H_i(\so_{n+1})_{\sigma}$ is surjective, then
\begin{equation*}
\Coker\{\rho_i\colon H_i(\so_{n+1})^{\sigma}\rightarrow H_i(\so_{n+1})_{\sigma}\}=0.
\end{equation*}
Adding lemma \ref{lemB}, we have
\begin{equation*}
\frac{H_i(\so_{n+1})}{H_i(\so_{n+1})^{\sigma}}\bigg/2\frac{H_i(\so_{n+1})}{H_i(\so_{n+1})^{\sigma}}\cong\Coker\{\rho_i\colon H_i(\so_{n+1})^{\sigma}\rightarrow H_i(\so_n)_{\sigma}\}=0.
\end{equation*}
Therefore $H_i(\so_{n+1})/H_i(\so_{n+1})^{\sigma}$ has neither torsion elements whose orders are devided by $2$ nor torsion-free elements.

Now we get the following:
\begin{lem}\label{lem10}
If the map $H_i(\so_n)_{\sigma}\rightarrow H_i(\so_{n+1})_{\sigma}$ is bijective, then
\begin{equation*}
\frac{H_i(\so_{n+1})}{H_i(\so_{n+1})^{\sigma}}\otimes_{\bb{Z}}\bb{Z}/2=0
\end{equation*}
\end{lem}
From lemma \ref{lemA}, the sufficient condition to satisfy the assumption of lemma \ref{lem10} is (\ref{onstab3}) and (\ref{sonstab3}).

We already have a result concerning to (\ref{onstab3}).
Recall that Cathelineau proved in \cite{cat2007} that
\begin{equation}\label{catsonstab}
\text{$H_i(\o_n)\rightarrow H_i(\o_{n+1})$ is bijective for $i<n$ and surjective for $i\leq n$}.
\end{equation}

Cathelineau proved in \cite{cat2007} that
\begin{equation}\label{catsonstr}
H_i(\so_n,\bb{Z}[1/2])_{\sigma}\cong H_i(\so_n,\bb{Z}[1/2])\quad\text{for $2i<n$}.
\end{equation}
Notice that with coefficient $\bb{Z}[1/2]$ we have $H_i(\so_n,\bb{Z}[1/2])^{\sigma}\cong H_i(\so_n,\bb{Z}[1/2])_{\sigma}$.

Because $\bb{Z}\left[1/2\right]$ is $\bb{Z}$-flat module,
we have that the sequence
\begin{equation*}
0\rightarrow H_i(\so_n)^{\sigma}\otimes_{\bb{Z}}\bb{Z}\left[1/2\right]\rightarrow H_i(\so_n)\otimes_{\bb{Z}}\bb{Z}\left[1/2\right]\rightarrow \left(\frac{H_i(\so_n)}{H_i(\so_n)^{\sigma}}\right)\otimes_{\bb{Z}}\bb{Z}\left[1/2\right]\rightarrow 0
\end{equation*}
is still exact.
Now we have that $H_i(\so_n)^{\sigma}\otimes_{\bb{Z}}\bb{Z}\left[1/2\right]\cong H_i(\so_n,\bb{Z}\left[1/2\right])^{\sigma}$ and
$H_i(\so_n)\otimes_{\bb{Z}}\bb{Z}\left[1/2\right]\cong H_i(\so_n,\bb{Z}\left[1/2\right])$, and therefore from the above exact sequence we get that
\begin{equation*}
\left(\frac{H_i(\so_n)}{H_i(\so_n)^{\sigma}}\right)\otimes_{\bb{Z}}\bb{Z}\left[1/2\right]\cong\frac{H_i(\so_n,\bb{Z}\left[1/2\right])}{H_i(\so_n,\bb{Z}\left[1/2\right])^{\sigma}}.
\end{equation*}

Hence we obtain the following reformulated Cathelineau's formula:
\begin{equation}\label{catssoninc}
\frac{H_i(\so_n)}{H_i(\so_n)^{\sigma}}\otimes_{\bb{Z}}\bb{Z}[1/2]=0\quad \text{for $2i<n$.} 
\end{equation}

We start the proof of theorem \ref{mainthm1} inductively.

First, the map
\begin{equation*}
H_0(\so_0)\xrightarrow{\cong} H_0(\so_1)\xrightarrow{\cong} H_0(\so_2)\xrightarrow{\cong}\cdots
\end{equation*}
are all bijective since we have $H_0(\so_n)=\bb{Z}$ for any $n$ and maps are natural.
Here we set $\so_0=\text{trivial group}$.

From (\ref{catsonstab}), the homological stability of $\o_n$ at $H_1$ is the following:
\begin{equation*}
0\cong H_1(\o_0)\rightarrow H_1(\o_1)\twoheadrightarrow H_1(\o_2)\xrightarrow{\cong}H_1(\o_3)\xrightarrow{\cong}\cdots.
\end{equation*}
We have, from lemma \ref{lemA}, the homological stability of $H_1(\so_n)_{\sigma}$ below:
\begin{equation*}
H_1(\so_1)_{\sigma}\rightarrow H_1(\so_2)_{\sigma}\xrightarrow{\cong}H_1(\so_3)_{\sigma}\xrightarrow{\cong}\cdots.
\end{equation*}
Thus from lemma \ref{lem10} we get that
\begin{equation}\label{mod2}
\frac{H_1(\so_m)}{H_1(\so_m)^{\sigma}}\bigg/2\frac{H_1(\so_m)}{H_1(\so_m)^{\sigma}}=0\text{ for $m>2$.}
\end{equation}
We also have from Cathelineau's theorem that
\begin{equation}\label{1/2attached}
\frac{H_1(\so_m)}{H_1(\so_m)^{\sigma}}\otimes_{\bb{Z}}\bb{Z}\left[1/2\right]=0\text{ for $m>2$.}
\end{equation}
Combining (\ref{mod2}) and (\ref{1/2attached}), we have that
\begin{equation}\label{total}
\frac{H_1(\so_m)}{H_1(\so_m)^{\sigma}}=0\text{ for $m>2$.}
\end{equation}
The condition (\ref{total}) means the existence of following isomorphisms:
\begin{equation*}
H_1(\so_m)^{\sigma}\xrightarrow{\cong} H_1(\so_m)\xrightarrow{\cong} H_1(\so_m)_{\sigma}.
\end{equation*}

Hence we get the following diagram.
\begin{equation*}
\xymatrix{
H_1(\so_2)_{\sigma}\ar[r]^{\cong} & H_1(\so_3)_{\sigma}\ar[r]^{\cong} & H_1(\so_4)_{\sigma}\ar[r]^{\cong} & \cdots \\
H_1(\so_2)\ar[r]\ar@{->>}[u] & H_1(\so_3)\ar[r]\ar[u]_{\cong} & H_1(\so_4)\ar[r]\ar[u]_{\cong} & \cdots \\
H_1(\so_2)^{\sigma}\ar@{^{(}->}[u] & H_1(\so_3)^{\sigma}\ar[u]\ar[u]_{\cong} & H_1(\so_4)^{\sigma}\ar[u]\ar[u]_{\cong} 
}
\end{equation*}
From the above diagram we have the stability of $H_1(\so_n)$ is
\begin{equation}\label{h1stabofso}
H_1(\so_2)\twoheadrightarrow H_1(\so_3)\xrightarrow{\cong} H_1(\so_4)\xrightarrow{\cong}\cdots
\end{equation}
and therefore we get that the map
\begin{equation*}
H_1(\so_n)\rightarrow H_1(\so_{n+1})
\end{equation*}
is surjective for $n\geq 2$ and bijective for $n>2$.

Next we see the homological stability of $H_2(\so_n)$.
From (\ref{catsonstab}), the homological stability of $\o_n$ at $H_2$ is the following:
\begin{equation*}
H_2(\o_2)\twoheadrightarrow H_2(\o_3)\xrightarrow{\cong} H_2(\o_4)\xrightarrow{\cong} \cdots.
\end{equation*}
Using this with the stability (\ref{h1stabofso}) of $H_1(\so_n)$, we have
\begin{equation*}
H_2(\so_3)_{\sigma}\xrightarrow{\cong} H_2(\so_4)_{\sigma}\xrightarrow{\cong}H_2(\so_5)_{\sigma}\xrightarrow{\cong} \cdots
\end{equation*}
from lemma \ref{lemA}.

Since we have the condition
\begin{equation*}
\frac{H_2(\so_m)}{H_2(\so_m)^{\sigma}}=0
\end{equation*}
for $m>4$ as in the case of $H_1$, again we get the diagram
\begin{equation*}
\xymatrix{
H_2(\so_4)_{\sigma}\ar[r]^{\cong} & H_2(\so_4)_{\sigma}\ar[r]^{\cong} & H_2(\so_4)_{\sigma}\ar[r]^{\cong} & \cdots \\
H_2(\so_4)\ar[r]\ar@{->>}[u] & H_2(\so_4)\ar[r]\ar[u]_{\cong} & H_2(\so_4)\ar[r]\ar[u]_{\cong} & \cdots \\
H_2(\so_4)^{\sigma}\ar@{^{(}->}[u] & H_2(\so_4)^{\sigma}\ar[u]_{\cong} & H_2(\so_4)^{\sigma}\ar[u]_{\cong}
}
\end{equation*}
Therefore we get that the map
\begin{equation*}
H_2(\so_n)\rightarrow H_2(\so_{n+1})
\end{equation*}
is surjective for $n\geq 4$ and bijective for $n>4$.

Now we finish the proof of main theorem \ref{mainthm1} by repeating this argument inductively.\qed

\begin{rem}\label{sonbound}
Cathelineau showed in \cite[Theorem 1.3 and Theorem 1.5]{cat2007} that if $F$ is a quadratically closed field,
the kernel of $H_n(\so_{2n},\bb{Z}\left[1/2\right])\rightarrow H_n(\so_{2n+1},\bb{Z}\left[1/2\right])$
is the $\bb{Z}\left[1/2\right]$-tensored Milnor $K$-group $K_n^{M}(F)\otimes\bb{Z}\left[1/2\right]$ of the field $F$.
\end{rem}

We refer to corollary \ref{strofson}.
Notice that in the proof of theorem \ref{mainthm1}, we also get a result about the structure of the homology group of special orthogonal groups.
Let us recall that when the degree of homology is equal to $n$, we get the following commutative diagram.
\begin{equation*}
\xymatrix{
H_n(\so_{2n})_{\sigma}\ar[r]^{\cong} & H_n(\so_{2n+1})_{\sigma}\ar[r]^{\cong} & H_n(\so_{2n+2})_{\sigma}\ar[r]^{\cong} & \cdots \\
H_n(\so_{2n})\ar[r]\ar@{->>}[u] & H_n(\so_{2n+1})\ar[r]\ar[u]_{\cong} & H_n(\so_{2n+2})\ar[r]\ar[u]_{\cong} & \cdots \\
H_n(\so_{2n})^{\sigma}\ar@{^{(}->}[u] & H_n(\so_{2n+1})^{\sigma}\ar[u]_{\cong} & H_n(\so_{2n+2})^{\sigma}\ar[u]_{\cong}
}
\end{equation*}
To prove that the bijectivity of the arrows $H_n(\so_m)^{\sigma}\rightarrow H_n(\so_m)$ for $m\geq 2n+1$,
we inductively proved and used the fact that $H_n(\so_m)/H_n(\so_m)^{\sigma}=0$ for $m\geq 2n+1$.
This implies that for $m\geq 2n+1$ the homology $H_n(\so_m)$ entirely consists of its $\sigma$-invariant part.

%% file: papier-6.tex
In this section we will prove theorem \ref{mainthm2} and get some applications.

\subsection{Proof of theorem \ref{mainthm2}}

As before we consider $\bb{Z}$ as $\bb{Z}\o_n$-module with trivial $\o_n$-action,
and $\bb{Z}^t$ as determinant-twisted $\bb{Z}\o_n$-module.
We consider the group ring $\bb{Z}[\bb{Z}/2]$ as an $\o_n$-module.
$\bb{Z}[\bb{Z}/2]$ has generators $\epsilon$ and $\sigma$
as $\bb{Z}$-module and we define an $\o_n$-action on $\bb{Z}[\bb{Z}/2]$ defined by
\begin{equation}
g\acts\epsilon=\begin{cases}\epsilon,\ \text{if $\det g=1$},\\ \sigma,\ \text{if $\det g=-1$},\end{cases}\quad g\acts\sigma=\begin{cases}\sigma,\ \text{if $\det g=1$, and}\\ \epsilon,\ \text{if $\det g=-1$}.\end{cases}
\end{equation}
Then we have the following two short exact sequences of $\o_n$-modules:
\begin{equation}\label{seq1}
1\rightarrow \bb{Z}^t\rightarrow\bb{Z}[\bb{Z}/2]\rightarrow \bb{Z}\rightarrow 1,
\end{equation}
where the left map sends $1$ to $\epsilon-\sigma$, and
\begin{equation}\label{seq2}
1\rightarrow \bb{Z}\rightarrow\bb{Z}[\bb{Z}/2]\rightarrow \bb{Z}^t\rightarrow 1,
\end{equation}
where the left map sends $1$ to $\epsilon+\sigma$.

Notice that from Shapiro's lemma (see \cite[Lemma 5.5]{dupont2001}) we have
\begin{equation}\label{shaplem}
H_i(\o_n,\bb{Z}[\bb{Z}/2])\cong H_i(\so_n)
\end{equation}
for any $i$ and $n$.

From the sequence (\ref{seq1}), we get the Bockstein long exact sequence
\begin{equation}\label{bseq1}
\cdots\rightarrow H_{i+1}(\o_n)\rightarrow H_i(\o_n,\bb{Z}^t)\rightarrow H_i(\so_n)\rightarrow H_i(\o_n)\rightarrow H_{i-1}(\o_n,\bb{Z}^t)\rightarrow\cdots,
\end{equation}
and from (\ref{seq2}) we get
\begin{equation}\label{bseq2}
\cdots\rightarrow H_{i+1}(\o_n,\bb{Z}^t)\rightarrow H_i(\o_n)\rightarrow H_i(\so_n)\rightarrow H_i(\o_n,\bb{Z}^t)\rightarrow H_{i-1}(\o_n)\rightarrow\cdots,
\end{equation}
respectively.

\begin{rem}
We can check easily that the map $H_i(\o_n)\rightarrow H_i(\so_n)$
induced from $\bb{Z}\rightarrow \bb{Z}[\bb{Z}/2]$
in \eqref{seq2} is the transfer map \cite[9 of Chapter III]{brown1982}.
\end{rem}

Let us see the following commutative diagram:
\begin{eqnarray*}
\xymatrix@C-15pt{
\cdots \ar[r] & H_{n+1}(\o_{2n},\bb{Z}^t) \ar[r] \ar[d] & H_{n+1}(\so_{2n}) \ar[r] \ar[d] & H_{n+1}(\o_{2n}) \ar[r] \ar[d] & H_n(\o_{2n},\bb{Z}^t) \ar[d] \\
\cdots \ar[r] & H_{n+1}(\o_{2n+1},\bb{Z}^t) \ar[r] & H_{n+1}(\so_{2n+1}) \ar[r] & H_{n+1}(\o_{2n+1}) \ar[r] & H_n(\o_{2n+1},\bb{Z}^t)
}\\
\xymatrix@C-15pt{
 \ar[r] & H_n(\so_{2n}) \ar[r] \ar[d] & H_n(\o_{2n}) \ar[r] \ar[d] & \cdots \\
 \ar[r] & H_n(\so_{2n+1}) \ar[r] & H_n(\o_{2n+1}) \ar[r] & \cdots
}
\end{eqnarray*}

where the horizontal sequences are (\ref{bseq1}) and the vertical maps are stability maps.

We already have stability results (proposition \ref{catprop} and theorem \ref{mainthm1}).
Therefore, using five lemma, we obtain theorem \ref{mainthm2}.

\begin{rem}
If the range can be expanded, from the Bockstein exact sequence (\ref{bseq2}),
we can expand the stability range of $H_{\ast}(\so_n)$.
But this may be false because the kernel of $H_n(\so_{2n})\rightarrow H_n(\so_{2n+1})$ may not be trivial (see remark \ref{sonbound}).
\end{rem}

\subsection{Some applications}

First we see the stability map $H_i(\inc,\bb{Z}^t)\colon H_i(\o_n,\bb{Z}^t)\rightarrow H_i(\o_{n+1},\bb{Z}^t)$.

Let $C_l$ be the free abelian group generated by the set
of all ordered $(l+1)$-tuples (denoted by $(v_0,\ldots,v_l)$) of points of $n$-dimensional unit sphere $S=S(F^{n+1})$
with the understanding that such an $l$-cell is zero if $v_0=v_1$,
and let $\partial$ be
\begin{equation*}
\partial(v_0,\ldots,v_l)=\sum_{j=0}^l(-1)^j(v_0,\ldots,\widehat{v_j},\ldots,v_l).
\end{equation*}
$\o_{n+1}$ acts diagonally on $C_l$.
$(C_{\ast},\partial)$ is a chain complex of $\o_{n+1}$-modules, which is acyclic with augmentation $\bb{Z}$.

We define a spectral sequence ${}^{\prime}E_{p,q}^1$ as the hyperhomology $H_p(\o_{n+1},C_q^t)$ which converges to the homology $H_{p+q}(\o_n,\bb{Z}^t)$.
Observe that
\begin{equation}\label{row0}
{}^{\prime}E_{i,0}^1\cong H_i(\o_n,\bb{Z}^t) \quad\text{for $i\geq 0$}.
\end{equation}
To see this, we use Shapiro's lemma; we have $H_i(\o_{n+1},C_0^t)\cong H_i(\Stab(v_0),(\bb{Z}(v_0))^t)$,
where $\Stab(v_0)$ is the stabilizer at $v_0\in S$, since $\o_{n+1}$ acts on $S$ transitively.
From Witt's theorem $\Stab(v_0)\cong \o_n$ for any $v_0$, and we get (\ref{row0}).

Next we calculate ${}^{\prime}E_{i,0}^2$.
From Shapiro's lemma we get
\begin{equation*}
{}^{\prime}E_{i,1}^1\cong \bigoplus_{(v_0,v_1)} H_i(\Stab((v_0,v_1)),\bb{Z}^t)\otimes \bb{Z}(v_0,v_1)
\end{equation*}
where $(v_0,v_1)$ runs all representatives of the set of $1$-cells decomposed to $\o_{n+1}$-orbits.
The differential $d^1\colon {}^{\prime}E_{i,1}^1\rightarrow {}^{\prime}E_{i,0}^1$ sends $c\otimes(v_0,v_1)$ to $c\otimes (v_1)-c\otimes(v_0)$,
where $c$ is an $i$-cycle of $\Stab((v_0,v_1))$.
From Shapiro's lemma in reverse, $c\otimes(v_i)$ represents an element of $H_i(\o_{n+1},C_0^t)$ for each $i=0,1$.
If $v_0$ and $v_1$ are linearly independent, then, from Witt's theorem, we can find an element $g\in \o_{n+1}$
which sends $v_0$ to $v_1$ and centralizes $\Stab((v_0,v_1))$.
Moreover we can find such element $g$ in $\so_{n+1}$ inverting its determinant if neccesary.
so that $c\otimes (v_1)-c\otimes(v_0)$ is homologous to zero.
If $v_0=-v_1$, then we get $c\otimes (v_1)-c\otimes(v_0)$ is homologous to $2c\otimes (v_1)$
because the determinant of the reflection of $v_0$ is $-1$ and this reflection centralizes $\Stab((v_0,v_1))$.
Therefore we obtain the following:
\begin{equation}\label{twisted2torsions}
{}^{\prime}E_{i,0}^2\cong H_i(\o_n,\bb{Z}^t)\otimes_{\bb{Z}} \bb{Z}/2 \quad \text{for $i\geq 0$}.
\end{equation}
Hence we obtain from (\ref{twisted2torsions}) that
\begin{equation}\label{quotfactor}
\text{$H_i(\o_n,\bb{Z}^t)\rightarrow H_i(\o_{n+1},\bb{Z}^t)$ factors through the quotient $H_i(\o_n,\bb{Z}^t)\otimes_{\bb{Z}} \bb{Z}/2$}.
\end{equation}

Applying this result (\ref{quotfactor}) to the stability theorem \ref{mainthm2}, we get the following result.
\begin{cor}
$H_i(\o_n,\bb{Z}^t)\cong H_i(\o_n,\bb{Z}^t)\otimes_{\bb{Z}}\bb{Z}/2$ for $2i<n$.
\end{cor}

Finally, we introduce a pair of useful results.
These results can be considered as the generalization of (\ref{catsonstr}).

The map $H_i(\so_n)\rightarrow H_i(\o_n)$ in Bockstein exact sequence (\ref{bseq1})
coincides the composite map $H_i(\so_n)\rightarrow H_i(\so_n)_{\sigma}\rightarrow H_i(\o_n)$
induced from Lyndon-Hochschild-Serre exact sequence associated to (\ref{ext}).
Recall that the map $H_i(\so_n)_{\sigma}\rightarrow H_i(\o_n)$ is injective (proved at (\ref{lem2})).
Then we get the following result from the sequence (\ref{bseq1}):
\begin{cor}
$\Coker\{H_i(\o_n,\bb{Z}^t)\rightarrow H_i(\so_n)\}\cong H_i(\so_n)_{\sigma}$ for $i\geq 0$,
where the map $H_i(\o_n,\bb{Z}^t)\rightarrow H_i(\so_n)$ is the map in (\ref{bseq1}).
\end{cor}

In the same way, from (\ref{bseq2}), we get the following:
\begin{cor}
$\Coker\{H_i(\o_n)\rightarrow H_i(\so_n)\}\cong (H_i(\so_n)^t)_{\sigma}$ for $i\geq 0$,
where $(H_i(\so_n)^t)_{\sigma}$ is determinant-twisted $\sigma$-invariant part.
\end{cor}